\documentclass{amsart}
\usepackage{amsfonts}
\usepackage{graphicx}
\usepackage{amscd}
\usepackage{amsmath}
\usepackage{amssymb}

\makeatletter
\@namedef{subjclassname@2010}{%
  \textup{2010} Mathematics Subject Classification}
\makeatother

\theoremstyle{plain}
\newtheorem*{acknowledgement}{Acknowledgement}

\newtheorem{lemma}{\bf Lemma}

\newtheorem{theorem}{\bf Theorem}

\numberwithin{equation}{section}

\begin{document}
\title[critical metrics on $4$-manifolds]{Critical metrics on $4$-manifolds\\ with harmonic anti-self dual Weyl tensor}

\author{Emanuel  Viana}
\address{Instituto Federal de Educa\c c\~ao, Ci\^encia e Tecnologia do Cear\'a - IFCE, Campus Caucaia, Caucaia / CE, Brazil.}
\email{emanuel.mendonca@ifce.edu.br}

\date{December 9, 2021}

\begin{abstract}
In this paper, we study $4$-dimensional simply connected, compact critical metric of the volume functional with harmonic anti-self dual Weyl tensor. We show that a $4$-dimensional simply connected, compact critical metric of the volume functional with harmonic anti-self dual Weyl tensor and satisfying  a suitable pinching condition is isometric to a geodesic ball in a simply connected space form $\mathbb{R}^{4}$, $\mathbb{H}^{4}$ or $\mathbb{S}^{4}.$ 
\end{abstract}

\maketitle

\section{Introduction}
\label{intro}

Let $(M^{n},\,g)$ be a connected compact Riemannian manifold with dimension $n$ at least three. Following the terminology adopted in \cite{BDR,BDRR}, we say that $g$ is, for brevity, {\it Miao-Tam critical metric} if there is a smooth function $f$ on $M^n$ such that $f^{-1}(0)=\partial M$ satisfying the overdetermined-elliptic system
\begin{equation}
\label{eqMiaoTam} \mathfrak{L}_{g}^{*}(f)=-(\Delta f)g+Hess\, f-fRic=g,
\end{equation} where $\mathfrak{L}_{g}^{*}$ is the formal $L^{2}$-adjoint of the linearization of the scalar curvature ope\-rator $\mathfrak{L}_{g},$ which plays a fundamental role in problems related to prescribing the scalar curvature function. Here, $Ric,$ $\Delta$ and $Hess$ stand, respectively, for the Ricci tensor, the Laplacian operator and the Hessian form on $(M^n,\,g).$ For more details on such a subject, we refer the reader to \cite{BalRi,BDR,BDRR,CEM,miaotam,miaotamTAMS,yuan}.

Miao and Tam \cite{miaotam} showed that such critical metrics arise as critical points of the volume functional on $M^n$ when restricted to the class of metrics $g$ with prescribed constant scalar curvature such that $g_{|_{T \partial M}}=h$ for a prescribed Riemannian metric $h$  on the boundary. In the work \cite{CEM}, Corvino, Eichmair and Miao also considered metrics arising from the modified problem of finding stationary points for the vo\-lume functional on the space of metrics whose scalar curvature is equal to a given constant, which is directly related to the Miao-Tam critical metrics. We point out that if there exists a nontrivial function $f$ satisfying \eqref{eqMiaoTam} then $g$ must have constant scalar curvature $R$ (see \cite[Theorem 7]{miaotam}). Explicit examples of Miao-Tam critical metrics were obtained on connected domain with compact closure in $\Bbb{R}^n,$ $\Bbb{H}^n$ and $\Bbb{S}^n$ (see \cite{miaotam}). Some results obtained in \cite{miaotam} suggest that critical metrics with a prescribed boundary metric seem to be rather rigid. From this perspective, it is natural to ask whether these quoted examples are the only simply connected Miao-Tam critical metrics with connected boundary.

It is well known that $4$-dimensional manifolds display fascinating and peculiar features, for this reason much attention has been given to this specific dimension. Many peculiar features on oriented $4$-manifolds directly rely on the fact that the bundle of $2$-forms $\Lambda^2$ can be invariantly decomposed as a direct sum $\Lambda^2=\Lambda_{+}^{2} \oplus \Lambda_{-}^{2}$ such that $$W=W^{+}\oplus W^{-}$$ where $W^{\pm}$ are called the {\it self-dual} and {\it anti-self dual} parts of $W,$ respectively. Half conformally flat metrics are also known as self-dual or anti-self dual if $W^{-} = 0$ or $W^{+} = 0$, respectively. In order to motivate our main result, we now briefly recall a few relevant results on the rigidity of Miao-Tam critical metrics in dimension $n=4$ under vanishing conditions involving zero, first, or specific second order derivatives of the Weyl tensor $W,$ which is defined by \eqref{weyl} in Section \ref{secThm1}. Initially, Miao and Tam \cite{miaotamTAMS} proved that a locally conformally flat  (i.e., $W=0$) $4$-dimensional simply connected, compact Miao-Tam critical metric $(M^{4},\,g,\,f)$ with boundary isometric to a standard sphere $\mathbb{S}^{3}$ must be isometric to a geodesic ball in a simply connected space form $\mathbb{R}^{4}$, $\mathbb{H}^{4}$ or $\mathbb{S}^{4}.$ Barros, Di\'{o}genes and Ribeiro \cite{BDR} proved that a Bach-flat (i.e., $B_{ij}=0,$ where $B_{ij}$ is given by \eqref{bach}) simply connected, $4$-dimensional compact Miao-Tam critical metric with boundary isometric to a standard sphere $\mathbb{S}^{3}$ must be isometric to a geodesic ball in a simply connected space form $\mathbb{R}^{4}$, $\mathbb{H}^{4}$ or $\mathbb{S}^{4}.$ This result includes, in particular, the half-conformally flat case (i.e., $W^+=0$). Recently, Kim and Shin \cite{Kim} showed that  a simply connected, compact Miao-Tam critical metric with harmonic Weyl tensor (i.e., ${\rm div} (W)=0$) and boundary isometric to a standard sphere $\mathbb{S}^{3}$ must be isometric to a geodesic ball in a simply connected space form $\mathbb{R}^{4}$, $\mathbb{H}^{4}$ or $\mathbb{S}^{4}$ (see also \cite{BBB,BalRi2} for an alternative proof).

For what follows, remember that viewing $W^{+}$ as a tensor of type $(0, 4)$, the tensor $W^+$ is harmonic if ${\rm{div}}\, (W^+) = 0,$ where ${\rm{div}}$ is the divergence defined for any $(0, 4)$-tensor $S$ by $${\rm{div}}\,S(X_1,X_2,X_3) = trace_{g}\{(Y,Z) \to \nabla_{Y}S(Z,X_1,X_2,X_3)\},$$ where $g$ is the metric of $M^4$. Furthermore, it should be emphasized that every $4$-dimensional Einstein manifold has harmonic tensor $W^+$. Therefore, it is natural to ask which geometric implications has the assumption of the harmonicity of the tensor $W^+$ on Miao-Tam critical metrics.

In this paper, we have established the following result. 

\begin{theorem}
\label{thmA}
Let $(M^{4},\,g,\,f)$ be a $4$-dimensional simply connected, compact Miao-Tam critical metric with boundary isometric to a standard sphere $\mathbb{S}^{3}.$ Suppose that $M^4$ has harmonic anti-self dual Weyl tensor (i.e., ${\rm{div}}\, (W^+) = 0$) and satisfies the inequality
\begin{equation}\label{ineth}
3|Ric|^2\leq 3R^2_{44}+(R-R_{44})^2,
\end{equation} where $R_{44}=Ric\left(\frac{\nabla f}{|\nabla f|},\frac{\nabla f}{|\nabla f|}\right).$ Then $(M^{4},\,g)$ is isometric to a geodesic ball in a simply connected space form $\mathbb{R}^{4}$, $\mathbb{H}^{4}$ or $\mathbb{S}^{4}.$
\end{theorem}

Clearly, if we change the condition ${\rm{div}}\,(W^+) = 0$ by the condition ${\rm{div}}\,(W^-) = 0$ the conclusion of Theorem \ref{thmA} is exactly the same.

\section{Background}
\label{secThm1}

In this section, we recall some basic tensors and informations that will be useful in the proof of main results.Throughout the paper the Einstein convention of summing over the repeated indices will be adopted.

Let $(M^n,\,g)$ a Riemannian manifold $n\geq 3.$ We recall that the Weyl tensor $W$ is defined by the following decomposition formula
\begin{eqnarray}
\label{weyl}
R_{ijkl}&=&W_{ijkl}+\frac{1}{n-2}\big(R_{ik}g_{jl}+R_{jl}g_{ik}-R_{il}g_{jk}-R_{jk}g_{il}\big) \nonumber\\
 &&-\frac{R}{(n-1)(n-2)}\big(g_{jl}g_{ik}-g_{il}g_{jk}\big),
\end{eqnarray} where $R_{ijkl}$ stands for the Riemann curvature operator. Moreover, the Cotton tensor $C$ is  given by
\begin{equation}
\label{cotton} \displaystyle{C_{ijk}=\nabla_{i}R_{jk}-\nabla_{j}R_{ik}-\frac{1}{2(n-1)}\big(\nabla_{i}R\,
g_{jk}-\nabla_{j}R\,g_{ik}).}
\end{equation} These two tensors are related as follows
\begin{equation}
\label{cottonwyel} \displaystyle{C_{ijk}=-\frac{(n-2)}{(n-3)}\nabla_{l}W_{ijkl},}
\end{equation} provided $n\ge 4.$ In particular, it is easy to see that
\begin{equation}\label{simetC}
C_{ijk}=-C_{jik} \,\,\,\,\,\,\hbox{and}\,\,\,\,\,\,C_{ijk}+C_{jki}+C_{kij}=0.
\end{equation} Also, we have
\begin{equation}\label{freetraceC}
g_{ij}C_{ijk}=g_{ik}C_{ijk}=g_{jk}C_{ijk}=0.
\end{equation} Next, the Schouten tensor $A$ is defined by
\begin{equation}
\label{schouten} A_{ij}=R_{ij}-\frac{R}{2(n-1)}g_{ij}.
\end{equation} Combining Eqs. (\ref{weyl}) and (\ref{schouten}) we obtain the following decomposition
\begin{equation}
\label{WS} R_{ijkl}=\frac{1}{n-2}(A\odot g)_{ijkl}+W_{ijkl},
\end{equation} where $\odot$ is the Kulkarni-Nomizu product.  In addition, we recall that the Bach tensor on a Riemannian manifold $(M^n,\,g)$, $n\geq 4,$ is defined in term of the components of the Weyl tensor $W_{ikjl}$ as follows
\begin{equation}
\label{bach} B_{ij}=\frac{1}{n-3}\nabla_{k}\nabla_{l}W_{ikjl}+\frac{1}{n-2}R_{kl}W_{ikjl},
\end{equation} while for $n=3$ it is given by
\begin{equation}
\label{bach3} B_{ij}=\nabla_{k}C_{kij}={\rm{div}}\,(C)_{ij}.
\end{equation} For more details about these tensors, we refer to \cite{bach,besse}. We say that $(M^n,\,g)$ is Bach-flat when $B_{ij}=0.$ It is not difficult to check that if $(M^n,\,g)$ is either locally conformally flat or Einstein, then $(M^n,\,g)$ is Bach-flat.

For what follows, $M^4$ will denote an oriented $4$-dimensional manifold and $g$ is a Riemannian metric on $M^4$. As was pointed out, $4$-dimensional manifolds are special. For instance, given any local orthogonal frame $\{e_1, e_2, e_3, e_4\}$ on an open set of $M^4$ with dual basis $\{e^1, e^2, e^3, e^4\},$ there exists a unique bundle morphism $*$ called \textit{the Hodge star} (acting on bivectors) such that $$*(e^1 \wedge e^2)=e^3 \wedge e^4.$$ This implies that $*$ is an involution, that is, $*^2 = Id$. In particular, this implies that the bundle of $2$-forms on a $4$-dimensional oriented Riemannian manifold can be invariantly decomposed as a direct sum $\Lambda^2=\Lambda_{+}^{2} \oplus \Lambda_{+}^{2}$. Thus, one concludes that the Weyl tensor $W$ is an endomorphism of $\Lambda^2$ such that
\begin{equation}\label{W4d}
	W=W^{+}\oplus W^{-}.
\end{equation}

We recall that $dim_{\mathbb{R}}(\Lambda^2) = 6$ and $dim_{\mathbb{R}}(\Lambda^{\pm}) = 3$. Since the Weyl tensor is trace-free on any pair of indices, one obtains that
\begin{equation}\label{WT4}
	W^{+}_{pqrs}=\dfrac{1}{2}(W_{pqrs}+W_{pq\overline{r}\,\overline{s}}),
\end{equation} where $(\overline{r}\,\overline{s})$, for instance, stands for the dual of $(r\,s)$, that is, $(r\, s\,\overline{r}\,\overline{s}) = \sigma(1234)$ for some even permutation $\sigma$ in the set $\{1, 2, 3, 4\}$ (see \cite{Dillen}, Equation 6.17]). In particular, we have $$W^{+}_{1234}=\dfrac{1}{2}(W_{1234}+W_{1212}).$$

Before proceeding, we remember that a Riemannian manifold $(M^n,\,g)$ is a Miao-Tam critical metric if there exists a smooth function $f$ such that \eqref{eqMiaoTam} holds, which in local coordinates is expressed by
\begin{equation}
	\label{MTcoord}
	-(\Delta f)g_{ij}+\nabla_i \nabla_j\,f-f R_{ij}= g_{ij}.
\end{equation} In particular, for arbitrary dimension $n,$ tracing \eqref{MTcoord} we deduce that the function potential $f$ also satisfies the linear equation
\begin{equation}
\label{eqtrace}
(1-n)\Delta f-Rf=n
\end{equation} and 
\begin{equation}
\label{covMTcoord}
		-(\Delta f)\nabla_{i}f+\dfrac{1}{2}\nabla_i|\nabla f|^2 -f R_{ij}\nabla_{j}f=\nabla_{i}f.
\end{equation} Taking the covariant derivative of \eqref{eqtrace}, we get
\begin{equation}
	\label{coveqtrace}
	\nabla_{i}R\, f+R\nabla_{i}f+(n-1)\nabla_{i} \Delta f=0.
\end{equation} Furthermore, taking the covariant derivative in \eqref{MTcoord}, we obtain
\begin{equation}
\label{covMTcoord1}
-\nabla_{k} (\Delta f)g_{ij}+\nabla_{k} \nabla_i \nabla_j\,f-\nabla_{k}f R_{ij}-f\nabla_{k} R_{ij}= 0.
\end{equation} Moreover, by Eqs. \eqref{MTcoord} and \eqref{eqtrace}, it is easy to check that
\begin{equation}
\label{p1a}
f\mathring{Ric}=\mathring{Hess\,f},
\end{equation} where $\mathring{S}=S-\dfrac{{\rm tr}\,S}{n}g$ stands for the traceless of tensor $S.$ 

We now recall some useful properties on Miao-Tam critical metrics obtained by Barros, Di\'ogenes e Ribeiro \cite{BDR}. Indeed, since a Miao-Tam critical metric has necessarily constant scalar curvature (see \cite{miaotam}), they used  \eqref{MTcoord} to obtain the following lemma.

\begin{lemma}
	\label{lem1} 
	Let $(M^n,\,g,\,f)$ be a Miao-Tam critical metric. Then
	\begin{equation*}
		f\big(\nabla_{i}R_{jk}-\nabla_{j}R_{ik}\big)=R_{ijkl}\nabla_{l}f + \frac{R}{n-1}\big(\nabla_{i}f
		g_{jk}-\nabla_{j}f g_{ik}\big)- \big(\nabla_{i}f R_{jk}-\nabla_{j}f R_{ik}\big).
	\end{equation*}
\end{lemma}

Barros, Di\'ogenes and Ribeiro also introduced in \cite{BDR} the covariant $3$-tensor $T$ given by
\begin{eqnarray}
	\label{T}
	T_{ijk}&=&\frac{n-1}{n-2}\left(R_{ik}\nabla_{j}f-R_{jk}\nabla_{i}f\right)
	-\frac{R}{n-2}\left(g_{ik}\nabla_{j}f-g_{jk}\nabla_{i}f\right)\nonumber\\
	&&+\frac{1}{n-2}\left(g_{ik}R_{js}\nabla_{s}f-g_{jk}R_{is}\nabla_{s}f\right).
\end{eqnarray} It is important to highlight that $T_{ijk}$ was defined similarly to $D_{ijk}$ in \cite{CaoChen}. With this notation, they proved the following lemma.
\begin{lemma}\label{lemTCW} 
	Let $(M^n,g,f)$ be a Miao-Tam critical metric. Then the following identity holds:
	\begin{equation}\label{fCTW}
		fC_{ijk}=T_{ijk}+W_{ijkl}\nabla_{l}f.
	\end{equation}
\end{lemma}

Now we ready to present the proof of Theorem \ref{thmA}.

\section{Proof of Theorem \ref{thmA}}
\label{proofA}

In this section, we shall present the proof of Theorem 1. To this end, we need to establish the following ingredient.
\begin{lemma}
	\label{keyresult}
	Let $(M^4, \,g)$ be a $4$-dimensional Miao-Tam metric satisfying \eqref{MTcoord} with ${\rm div}\, (W^+)=0$, and let $\Sigma_c=\{x\in M^4;f(x)=c\}$ be the level surface with respect to regular value $c$ of $f$. Then, for any local orthonormal frame $\{e_1,e_2,e_3,e_4\}$ with $e_4=\dfrac{\nabla f}{|\nabla f|}$ and  $\{e_1,e_2,e_3\}$ tangent to $\Sigma_c$, we have
		\begin{enumerate}
			\item $\nabla f$, whenever nonzero, is an eigenvector for $Ric$.
			\item $R_{4i}=0$ for any $i\in \{1,2,3\}$, and $e_4$ is an eigenvector of Ricci operator.
			\item $|\nabla f|^2$ are constants on $\Sigma_c$.
			\item The second fundamental form and the mean curvature of $\Sigma_c$ are, respectively,
				$$h_{ij}=\dfrac{(1+\Delta f)g_{ij}+fR_{ij}}{|\nabla f|}$$
			and
				$$H=\dfrac{3(1+\Delta f)+f(R-R_{44})}{|\nabla f|}.$$
			\item Around any regular point of $f,$ it holds
				\begin{equation}
					\label{ine1}
					3|Ric|^2\geq 3R^2_{44}+(R-R_{44})^2.
				\end{equation} Moreover, equality holds in (\ref{ine1}) if and only if any point of $\Sigma_c$ is umbilical. Also, $H$ is constant on $\Sigma_c$ and $W_{4ijk}=0$ for $i,j,k\in \{1,2,3\}$.
		\end{enumerate}
\end{lemma}

\begin{proof}
		To begin with, we shall compute the value of ${\rm div}\, (W^{+})$ on Miao-Tam critical manifolds satisfying \eqref{MTcoord}. By Eqs. \eqref{cottonwyel} and \eqref{WT4}, we get
		\begin{eqnarray*}
			\nabla_{i} W^{+}_{ijkl}=\dfrac{1}{2}\left(\nabla_{i} W_{ijkl}+\nabla_{i} W_{ij\overline{k}\,\overline{l}}\right)
		\end{eqnarray*}
		and
		\begin{eqnarray*}
			\nabla_{i} W_{ijkl}=\dfrac{1}{2}C_{klj} \mbox{ and } \nabla_{i} W_{ij\overline{k}\,\overline{l}}=\dfrac{1}{2}C_{\overline{k}\,\overline{l}j},
		\end{eqnarray*} 
		so that
		\begin{equation}\label{iW}
			4\nabla_{i} W^{+}_{ijkl}=C_{klj}+C_{\overline{k}\,\overline{l}j}.
		\end{equation} Therefore, our assumption yields
		\begin{equation}\label{C}
			C_{klj}+C_{\overline{k}\,\overline{l}j}=0.
		\end{equation} Moreover, by \eqref{fCTW} and \eqref{C}, we obtain
		\begin{equation}\label{meq}
			f(C_{ijk}+C_{\overline{i}\,\overline{j}k})=0=T_{ijk}+T_{\overline{i}\,\overline{j}k}+(W_{ijkl}+W_{\overline{i}\,\overline{j}kl})\nabla_{l}f.
		\end{equation}
		
		From the anti-symmetry of the Weyl tensor and (\ref{meq}), one sees that
		\begin{equation}\label{T0}
			(T_{ijk}+T_{\overline{i}\,\overline{j}k})\nabla_{k}f=0.
		\end{equation}
		Next, since $\nabla_{i}f=g_{ij}\nabla_{j}f$, we may use \eqref{T} to infer
		\begin{eqnarray*}
			T_{ijk}\nabla_{k}f&=&\frac{3}{2}\left(R_{ik}\nabla_{j}f\nabla_{k}f-R_{jk}\nabla_{i}f\nabla_{k}f\right)\\
			&&-\frac{R}{2}\left(g_{ik}\nabla_{j}f\nabla_{k}f-g_{jk}\nabla_{i}f\nabla_{k}f\right)\\
			&&+\frac{1}{2}\left(g_{ik}R_{js}\nabla_{s}f\nabla_{k}f-g_{jk}R_{is}\nabla_{s}f\nabla_{k}f\right)\\
			&=&\frac{3}{2}\left(R_{ik}\nabla_{j}f\nabla_{k}f-R_{jk}\nabla_{i}f\nabla_{k}f\right)\\
			&&+\frac{1}{2}\left(R_{jk}\nabla_{k}f\nabla_{i}f-R_{ik}\nabla_{k}f\nabla_{j}f\right),
		\end{eqnarray*} 
		consequently,
		
		\begin{equation}\label{Tkf}
			T_{ijk}\nabla_{k}f=R_{ik}\nabla_{j}f\nabla_{k}f-R_{jk}\nabla_{i}f\nabla_{k}f.
		\end{equation} 
		
		Therefore, from \eqref{T0} and \eqref{Tkf}, one obtains that
		\begin{equation*}
			(R_{ik}\nabla_{j}f\nabla_{k}f-R_{jk}\nabla_{i}f\nabla_{k}f)+(R_{\overline{i}k}\nabla_{\overline{j}}f\nabla_{k}f-R_{\overline{j}k}\nabla_{\overline{i}}f\nabla_{k}f)=0,
		\end{equation*} 
		or equivalently,
		
		\begin{equation}\label{MainEquation}
			(R_{ik}\nabla_{j}f-R_{jk}\nabla_{i}f)\nabla_{k}f+(R_{\overline{i}k}\nabla_{\overline{j}}f-R_{\overline{j}k}\nabla_{\overline{i}}f)\nabla_{k}f=0.
		\end{equation}

In order to proceed, following \cite{BLRi,Benedito}, we need consider an orthonormal frame $\{e_1,e_2,e_3,e_4\}$ diagonalizing the Ricci curvature $Ric$ at a point $q,$ namely, $$Ric_q(e_i,e_j)=R_{ii}(q)\delta_{ij},$$ such that $\nabla f(q)\neq 0,$ with associated eigenvalues $R_{kk}$, $k\in \{1,2,3,4\}$, res\-pectively. We highlight that the regular points of $\{x\in M^4; (\nabla f)(x)\neq 0\}$ is dense on $(M^4, \, g),$ otherwise, since $f$ is real-analytic (see Proposition 2.1 in \cite{CEM}), $f$ must be constant in an open set  $M^4.$ Therefore, from now on, up to explicit mention, we restrict our attention to regular points.

Next, it follows from \eqref{MainEquation} the the following equations
		\begin{equation}\label{sist}
			\left\{\begin{array}{lll}
				\left(R_{11}-R_{22}\right)\nabla_1 f \nabla_2 f+\left(R_{33}-R_{44}\right)\nabla_3 f \nabla_4 f=0,\\
				\left(R_{11}-R_{33}\right)\nabla_1 f \nabla_3 f+\left(R_{44}-R_{22}\right)\nabla_4 f \nabla_2 f=0,\\
				\left(R_{11}-R_{44}\right)\nabla_1 f \nabla_4 f+\left(R_{22}-R_{33}\right)\nabla_2 f \nabla_3 f=0.\\
			\end{array} \right.
		\end{equation}

		Now, we are going to prove the first item of the lemma. Indeed, taking into account that $\nabla f\neq 0,$ one obtains that at least one of the $\left(\nabla_j f\right) \neq 0$, $1\leq j \leq 4.$ If this occurs for exactly one of them, then $\nabla f=\left( \nabla_j f \right) e_j$  for some $j,$  and therefore $$Ric(\nabla f)=R_{jj}\nabla f.$$ On the other hand, if we have $\left(\nabla_j f\right) \neq 0$ for two directions, without loss generality, we can suppose that $\nabla_1 f \neq 0$, $\nabla_2 f \neq 0$, $\nabla_3 f =0$ and $\nabla_4 f =0.$ Then, it follows from \eqref{sist} that $R_{11}=R_{22}=\lambda,$ and in such case, we infer $\nabla f=\left(\nabla_1 f\right)e_1+\left(\nabla_2 f\right)e_2$. Hence,
		\begin{eqnarray}\nonumber
			Ric (\nabla f)&=& Ric\left(\left(\nabla_1 f\right)e_1+\left(\nabla_2 f\right)e_2\right)\\
			&=&\left(\nabla_1 f\right)R_{11}e_1+\left(\nabla_2 f\right)R_{22}e_2=\lambda \nabla f.
		\end{eqnarray} The case $\left(\nabla_j f\right) \neq 0$ for three directions is analogous. Now, it remains to analyze the case $\left(\nabla_j f\right) \neq 0$ for $j=1,2,3,4.$ In this situation, we use once more \eqref{sist} to obtain
		\begin{eqnarray}\nonumber
			&&\left(R_{11}-R_{22}\right)^2(\nabla_1 f \nabla_2 f)^2+\left(R_{33}-R_{44}\right)^2(\nabla_3 f \nabla_4 f)^2\\ \nonumber
			&&+\left(R_{11}-R_{33}\right)^2(\nabla_1 f \nabla_3 f)^2+\left(R_{44}-R_{22}\right)^2(\nabla_4 f \nabla_2 f)^2\\ 
			&&+\left(R_{11}-R_{44}\right)^2(\nabla_1 f \nabla_4 f)^2+\left(R_{22}-R_{33}\right)^2(\nabla_2 f \nabla_3 f)^2=0.
		\end{eqnarray} Whence, $R_{11} = R_{22} = R_{33} = R_{44}.$ Therefore, $\nabla f$ is an eigenvector for $Ric$, which proves our claim. 

The second assertion follows from the fact that we can choose $e_4 = \dfrac{\nabla f}{|\nabla f|}$ and $\nabla f$ is eigenvector for $Ric.$ 

Again, since $\nabla f$ is an eigenvector of the Ricci tensor, by choosing an orthonormal frame $\{e_{1},\,e_{2},\,e_{3},\,e_{4}=\frac{\nabla f}{|\nabla f|}\},$ we have $\nabla_i f=g(e_i,\nabla f)=0$ for any $i\in \{1,2,3\}$ and therefore, it follows from \eqref{covMTcoord} that $\nabla_i|\nabla f|^2=0,$ $\forall i=1,2,3.$ Thus, $|\nabla f|^2$ is constant in $\Sigma_c,$ which proves to the third assertion.

Moreover, 
	\begin{eqnarray*}
		h_{ij}&=&g\left(\nabla_{i}\left(\dfrac{\nabla f}{|\nabla f|}\right),e_j\right)\\
		&=&\dfrac{1}{|\nabla f|}\nabla_i \nabla_j\,f.
	\end{eqnarray*}
Now, since $\nabla_i f=0$, from \eqref{MTcoord} we get
\begin{eqnarray*}
	h_{ij}&=&\dfrac{1}{|\nabla f|}((1+\Delta f)g_{ij}+fR_{ij}).
\end{eqnarray*} In particular, tracing the above relation over $i$ and $j$ one sees that $$H=\dfrac{3(1+\Delta f)+f(R-R_{44})}{|\nabla f|}.$$

Observe that
\begin{eqnarray*}
	0\leq \left|h_{ij}-\dfrac{H}{3}g_{ij}\right|^2&=&\left|\dfrac{3(1+\Delta f)g_{ij}+3fR_{ij}}{3|\nabla f|}-\dfrac{3(1+\Delta f)+f(R-R_{44})}{3|\nabla f|}g_{ij}\right|^2\\
	&=&\dfrac{1}{|\nabla f|^2}\left|\dfrac{f(R-R_{44})}{3}g_{ij}-fR_{ij}\right|^2.
\end{eqnarray*}

Therefore, by using the fact that $$|Ric|^2=|R_{ij}|^2+2\sum_{i=1}^{3}R_{i4}^2+R_{44}^2$$ and the first item, we obtain
\begin{eqnarray*}
	0\leq |\nabla f|^2\left|h_{ij}-\dfrac{H}{3}g_{ij}\right|^2&=&\left|\dfrac{f(R-R_{44})}{3}g_{ij}-fR_{ij}\right|^2\\
	&=&f^2\left[3\left(\dfrac{R-R_{44}}{3}\right)^2-2\dfrac{(R-R_{44})^2}{3}+|R_{ij}|^2\right]\\
	&=&f^2\left(|Ric|^2-R^2_{44}-\dfrac{(R-R_{44})^2}{3}\right)
\end{eqnarray*} which proves \eqref{ine1}.

Suppose that the equality holds in above expression, then every point of $\Sigma_c$ is umbilical, that is, $$h_{ij}=\dfrac{H}{3}g_{ij}.$$ So, the Codazzi equation give us, for $i,j,k\in \{1,2,3\},$ $$R_{4ijk}=\nabla_{k}h_{ij}-\nabla_{j}h_{ik},$$ and thereby,
\begin{eqnarray*}
	0&=&R_{4j}\\
	&=&\nabla_{k} h_{kj}-\nabla_{j}H\\
	&=&\nabla_{k} \dfrac{H}{3}g_{kj}-\nabla_{j}H\\
	&=&\dfrac{1}{3}\nabla_{j} H-\nabla_{j}H\\
	&=&-\dfrac{2}{3}\nabla_{j}H,
\end{eqnarray*} i.e., $H$ is constant in $\Sigma_c$. Whence,

\begin{eqnarray*}
	R_{4ijk}&=&\nabla_{k} h_{ij}-\nabla_{j}h_{ik}\\
	&=&\nabla_{k} \dfrac{H}{3}g_{ij}-\nabla_{j}\dfrac{H}{3}g_{ik}\\
	&=&\dfrac{H}{3}(\nabla_{k}\,g_{ij}-\nabla_{j}\,g_{ik})\\
	&=&0.
\end{eqnarray*}

Finally, since $\nabla_{i} f=0 $ and $R_{4i}=0$  for $i\in \{1,2,3\},$ we have, from the decomposition formula of the Weyl tensor,
\begin{eqnarray*}
	R_{4ijk}&=&W_{4ijk}+\dfrac{1}{2}\left(R_{4j}g_{ik}+R_{ik}g_{4j}-R_{4k}g_{ij}-R_{ij}g_{4k}\right)\\
	&&-\dfrac{R}{6}\left(g_{ik}g_{4j}-g_{4k}g_{ij}\right),
\end{eqnarray*}and then $W_{4ijk}=0$ $i,j,k\in \{1,2,3\}.$ This finishes the proof of the lemma. 

\end{proof}

Now, we are going to finish the proof of Theorem \ref{thmA}.

\subsection{Conclusion of the Proof of Theorem \ref{thmA}}

\begin{proof}

To complete the proof of the theorem we follow \cite{BLRi}. Indeed, we already know that the set $\{x\in M^4; (\nabla f)(x)\neq 0\}$ is dense on $(M^4, \, g).$ In particular, \eqref{ineth} and the fifth item of Lemma \ref{keyresult} imply that $W(e_i,e_j,e_k,\nabla f)=0$ for $i,\,j,\,k\in \{1,2,3\}$. Since $W(e_i,e_j,e_k,\nabla f)=0$  and this jointly with the second item of Lemma \ref{keyresult}, \eqref{T} and \eqref{fCTW} yields $C_{ijk}=0.$ 

Proceeding, since ${\rm div}\, (W^+)=0$ and $C_{ijk}=0,$ it follows from \eqref{iW} that $C_{\overline{i}\,\overline{j}k}=0,$ i.e.,
\begin{equation}
	0=C_{\overline{1}\,\overline{2}k}=C_{34k},\,\,\, 0=C_{\overline{1}\,\overline{3}k}=-C_{24k}\,\,\, \mbox { and }\,\,\, 	0=C_{\overline{2}\,\overline{3}k}=C_{14k}.
\end{equation}

By the symmetric properties of the Weyl tensor, the first item of Lemma \ref{keyresult} and \eqref{fCTW}, one obtains that $C_{ij4}=0$, with $i,j\in \{1,2,3,4\}$. Also, since the Cotton tensor is skew-symmetric, we infer $C_{ijk}=0$, for $i,j,k\in \{1,2,3,4\}.$ Whence, one sees from \eqref{cottonwyel} that ${\rm div}\,(W)=0$ in $(M^4,\,g).$ Finally, it suffices to apply Theorem 10.2 in \cite{Kim} to conclude that the manifold $M^4$ is isometric to a geodesic ball in a simply connected space form $\mathbb{R}^{4}$, $\mathbb{H}^{4}$ or $\mathbb{S}^{4}.$ So, the proof is finished. 
 
\end{proof}

\begin{acknowledgement} The author want to thank the referee for the careful reading, relevant remarks and valuable suggestions.
\end{acknowledgement}

\end{document}